\documentclass[a4paper, 12pt,reqno]{amsart}

\usepackage[T1]{fontenc}
\usepackage[utf8]{inputenc}
\usepackage[english]{babel}
\usepackage{url}
\usepackage{ulem}
\usepackage{xcolor}
\usepackage{amsmath}
\usepackage{amssymb}
\usepackage{amsthm}
\usepackage{mathtools}
\usepackage[italian]{varioref}
\usepackage{hyperref}
\usepackage{typehtml}
\usepackage{multicol}
\usepackage{color}
\usepackage{enumerate}
\usepackage{mathrsfs}
\usepackage{comment}
\usepackage{mathabx}
\usepackage{tikz-cd}
\usepackage{adjustbox}
\usepackage{bbm}
\usepackage{hyperref}
\usepackage[a4paper,top=3cm,bottom=3cm,left=3.5cm,right=3.5cm]{geometry}

\usepackage{scalerel,stackengine}
\stackMath
\newcommand\reallywidehat[1]{%
\savestack{\tmpbox}{\stretchto{%
 \scaleto{%
    \scalerel*[\widthof{\ensuremath{#1}}]{\kern-.3pt\bigwedge\kern-.3pt}%
    {\rule[-\textheight/2]{2ex}{\textheight}}
  }{\textheight}%
}{1ex}}%
\stackon[1pt]{#1}{\tmpbox}%
}
\hypersetup{hidelinks}

\newcommand{\R}{\mathbb R}

\newcommand{\N}{\mathbb N}

\newcommand{\e}{\mathrm{e}}                
\newcommand{\de}{\mathop{}\!\mathrm{d}}

\newcommand{\eps}{\varepsilon}

\newcommand{\fcr}{\mathcal X}
\newcommand{\disc}{{\mathscr{D}}}
\newcommand{\wass}{{\mathscr{W}}}

\def\ie{\textit{i.e.} }
\def\be{\begin{equation}}
\def\ee{\end{equation}}
\def\bea{\begin{eqnarray}}
\def\eea{\end{eqnarray}}

\makeatletter
\@namedef{subjclassname@2020}{\textup{MSC2020} subject classifications}
\makeatother

\newtheorem{oss}{Remark}

\newtheorem{lem}{Lemma}
\newtheorem{prop}{Proposition}
\newtheorem{thm}{Theorem}

\numberwithin{equation}{section}

\title[Propagation of chaos and hydrodynamic description for topological models]
{Propagation of chaos and hydrodynamic description for topological models}

\author[D.\ Benedetto]{Dario Benedetto}
\address{Dario Benedetto \hfill\break \indent
	Dipartimento di Matematica, Universit\`a di Roma `La Sapienza'
	\hfill\break \indent
	P.le Aldo Moro 2, 00185 Roma, Italy}
\email{benedetto@mat.uniroma1.it}

\author[T.\ Paul]{Thierry Paul}
\address{Thierry Paul \hfill \break \indent
CNRS Laboratoire Ypatia des Sciences Mathematiques (LYSM), Rome, Italy }
\email{thierry.paul@sorbonne-universite.fr}

\author[S.\ Rossi]{Stefano Rossi}
\address{Stefano Rossi\hfill\break \indent
          Institut für Mathematik, Universität Zürich
	\hfill\break \indent
Winterthurerstrasse 190,
8057 Zürich, Switzerland}
\email{stefano.rossi@math.uzh.ch}

\begin{document}

\begin{abstract}
In this work, we study the deterministic Cucker-Smale model with topological interaction.
Focusing on the solutions of the corresponding Liouville equation, we show that propagation of chaos holds.
Moreover, considering monokinetic solutions, 
we also obtain a rigorous derivation of the hydrodynamic description given by a pressureless 
Euler-type system.

\end{abstract}

\keywords{propagation of chaos, monokinetic solutions, topological interaction, Cucker-Smale model}
\subjclass[2020]{
	35Q92, 
	35Q83, 
	82B40, 
}

\maketitle

\tableofcontents

\section{Introduction}

In recent decades, physics of complex systems has increasingly dealt with
the description of groups of animals exhibiting collective behavior, 
such as flocks of birds, fish schools, locust swarms, and migrating cells (\cite{E},\cite{CS},\cite{LGCT},\cite{BBHA},\cite{GBKM}).

From a modeling point of view, 
these systems offer new challenges and various models have been proposed to describe 
their interaction (see for example \cite{VCBCS},\cite{BDG},\cite{VZ}).
Among the first to be introduced,
the Cucker-Smale model (\cite{CS1}) describes a bird as a self-propelling particle 
interacting with its neighbors. 
In this case, the interaction is such that neighboring birds
tend to align their velocities and
the strength of the interaction is described through
weights which depend on the metric distance between the agents.

Around $2008$, a new type of interaction between agents called ``topological interaction'' was introduced.
In \cite{BCCC,CCGP}, the CoBBS group in Rome, after collecting $3$D observational data for flocks of starlings, observed that regardless of the density of the flock, 
each agent interacts on average with its first $6$-$7$ neighbors. 
This suggests that the strength of the interaction
between agents
does not depend on the metric distance between them, 
but rather on the ``topological'' distance that takes into account the proximity rank of the latter with respect to the former (see also \cite{BFW,CCGPS,GC,NMG, SB, Ma,SB2,WC}).

In the following years, deterministic and stochastic models with topological interactions were introduced in the mathematical physics literature.
In \cite{BD, BD2, DP, DPR} kinetic models of Boltzmann type are derived for topological interaction models based on jump processes.
As far as deterministic models are concerned, in \cite{H} a Cucker-Smale model is introduced where the interaction, instead of being metric as usual,
is topological. 
In \cite{H}, the kinetic mean-field and
hydrodynamic equations of Euler type are also written and
derived in
the case
of a smoothed version of the model.
Indeed, from a mathematical point of view, topological
interactions fall outside the case of two-body interaction
and present various problems in the derivation
of mean-field and hydrodynamic equations.
Specifically, the continuity estimate à la Dobrushin (see \cite{D}),
valid in the metrical 
case of regular two-body interactions,
does not work here since the solutions are not weakly continuous with respect to the initial data, as shown in \cite{BCR}. In \cite{BCR} the existence
of the dynamics and the mean-field limit have been rigorously proved for this same model, considering solutions of the limit equation with bounded density and proving that, for positive times $t$, the Wasserstein distance between the limit solution and the empirical measure at time $t$ tends to zero as the number of particles increases, if this holds at the initial time.

In the present work, we focus on the problems of propagation of chaos and
derivation of the hydrodynamic equations for this model.
The starting point is the Liouville equation verified by the $N$-particle
system and the aim is to show that the marginals of the $N$-body distribution
function converge to tensorial powers of solutions
of a suitable kinetic equation of Vlasov type.
With the same approach, we show the validity of
the hydrodynamic equations which describe the evolution
of monokinetic initial data.
The analysis will follow the approach
in \cite{GMP} (see also \cite{NP}, \cite{NP0}, \cite{MNP} for numerical considerations and \cite{PT} for the case of agent systems).

Due to the aforementioned difficulties,
the topological nature of the problem requires different techniques
for the proof of these two results. Regarding propagation of chaos,
the analysis makes use of another distance between probability measures,
called discrepancy distance, in addition to the Wasserstein one.
Furthermore, being the dynamics well-defined only for
almost all initial data, we need to consider a regularization of
the monokinetic initial datum in the proof of the hydrodynamic derivation,
obtaining the result for any limiting point with respect to the regularization.

The outline of the paper is as follows: in Section \ref{sec:2} we
recall the topological Cucker-Smale model and introduce the associated
Liouville equation, as well as recalling the results obtained in
\cite{BCR} which will be useful later on. In Section \ref{sec:3} we
provide a proof of the propagation of chaos, which will be a direct
consequence of the validity of the law of large numbers. In Section
\ref{sec:4} we focus on the derivation of the hydrodynamic description
given by Euler-type equations studying the so-called monokinetic
solutions.

\section{Model and general framework}
\label{sec:2}

A Cucker-Smale type model for the motion of $N$ agents,
in the mean-field scaling, is the system
\begin{equation}
  \label{eq:CS}
  \left\{
  \begin{aligned}
    &\dot x_i(t)=v_i(t)\\
    &\dot v_i(t)= \frac 1N \sum_{j=1}^N p_{ij} (v_j(t) - v_i(t)), \quad i=1, \dots, N
  \end{aligned}
  \right.
\end{equation}
where $(x_i, v_i) \in \R^d \times \R^d$ ($d=1,2,3,\dots$) and
the ``communication weights'' $\{p_{ij}\}_{i,j=1}^N$
are positive functions that take into account 
the interactions between agents. In classical models,
$p_{ij}$ depends only on the
euclidean distance
$|x_i - x_j|$
between the agents.
In topological models the weights depend on the positions of the agents
through their rank:
\begin{equation}
  \label{eq:pij}
  p_{ij} \coloneq K\bigl(M(x_i,|x_i - x_j|)\bigr),
\end{equation}
where $K
\colon [0,1] \to \R^+$ is a positive Lipschitz continuous non-increasing
function,
and, for $r>0$, the rank function
\begin{equation}
  \label{eq:Mpart}
  M(x_i,r) \coloneq \frac 1N \sum_{k=1}^N \fcr\{|x_k - x_i| \le r\}
\end{equation}
counts the number of agents at a distance
less than or equal to $r$ from $x_i$, normalized with $N$.
Note that in this case $p_{ij}$ is a piecewise-constant function of the positions
of all the agents.

We indicate by $\mathcal{P}(\R^k)$ the space of probability measures on $\R^k$. 
In the mean-field limit $N\to +\infty$, the one-agent distribution
function $f_t = f(t,x,v)\in \mathcal{P}(\R^{2d})$
is expected to verify the equation 
\begin{equation}
\label{eq:principale}
\partial_t f(t,x,v) + v \cdot \nabla_{x} f(t,x,v) +
\nabla_v \cdot \left(
W[S f_t,f_t](x,v)
f(t,x,v)\right)=0,
\end{equation}
where $Sf_t(x)\coloneq\int f_t(x,v) \de v \in \mathcal P(\R^d)$
denotes the spatial distribution
and, given a probability measure $f\in \mathcal P(\R^d\times \R^d)$
and a probability measure $\rho\in \mathcal P(\R^d)$,
$W[\rho,f]$ is the mean-field interaction given by
\begin{equation}
\label{eq:interazione}
W[\rho, f](x,v)\coloneq \int_{\R^d \times \R^d} K\left(M[\rho](x,|x-y|)\right)\,(w-v)f(y,w) \, \de y \, \de w,
\end{equation}
with
\begin{equation}
  \label{eq:Mrho}
  M[\rho](x,r)\coloneq\int_{{\bar B}_r(x)} \de \rho.
\end{equation}
Here and after, ${\bar B}_r(x)$ denotes the closed
ball in $\R^d$ of center $x$ and radius $r$.
We also indicate by ${\bar B}_R$ the closed ball ${\bar B}_R(0)$. 
Note that $ M[\rho](x,r)\le 1$ for any $x \in \R^d$ and $r\ge0$.

A weak formulation of equation \eqref{eq:principale} is given
requiring that the solution $f_t$ fulfills the identity
\begin{equation}
\label{pushforward}
\int \alpha(x,v) \de f_t(x, v) =
\int \alpha\left(x(t,x,v),v(t,x,v)\right) \de f_0(x,v)
\end{equation}
for any $\alpha \in C_b(\R^d\times \R^d)$, 
where $f_0$ is the initial probability measure 
and $(x(t,x,v),v(t,x,v))$ is the solution of the following
Cauchy
problem 
\begin{equation}
  \label{eq:flusso}
  \left\{
    \begin{aligned}
      &\dot x(t,x,v)=v(t,x,v)\\ 
      &\dot v(t,x,v)=W[Sf_t,f_t](x(t,x,v),v(t,x,v))  \\
      & x(0,x,v) = x,\ \ v(0,x,v) = v.\\
    \end{aligned}
  \right.
\end{equation}
In other words, $f_t$ is the push-forward
of $f_0$ along the flow generated by the one-particle system
\eqref{eq:flusso}, where the force $W$ depends 
on 
 $f_t$ itself.

Given
$$
Z_N=(x_1, \dots, x_N, v_1, \dots, v_N)\in \R^{dN} \times \R^{dN},
$$
we define the empirical measure on $\R^d \times \R^d$ as
\begin{equation}
\label{empirica}
  \mu_{Z_N} \coloneq
  \frac{1}{N}
  \sum_{i=1}^N \delta_{x_i} \otimes \delta_{v_i}.
\end{equation}
It is easy to verify that if $Z_N(t)=(x_1(t), \dots, x_N(t), v_1(t), \dots, v_N(t))$ solves
\eqref{eq:CS}, \eqref{eq:pij}, \eqref{eq:Mpart}, then $\mu_{Z_N(t)}$ 
is a weak solution of \eqref{eq:principale}.
Namely,
$M[S\mu_{Z_N}](x,r)$ is exactly 
$M(x,r)$ defined in equation \eqref{eq:Mpart} above.
Thus
we can rewrite the agent evolution for
the Cucker-Smale model \eqref{eq:CS}
with topological interactions
\eqref{eq:pij}, \eqref{eq:Mpart}
as 
\begin{equation}
  \label{eq:agenti}
  \left\{
    \begin{aligned}
      &\dot x_i(t)=v_i(t)\\
      &\dot v_i(t)= W[S\mu_{Z_N(t)},\mu_{Z_N(t)}](x_i(t),v_i(t)), \quad i=1, \dots, N.
    \end{aligned}
  \right.
\end{equation}
We indicate by $Z_N(t,Z_N) = (X_N(t,Z_N), V_N(t,Z_N))$ the solution
of this system
with initial datum $Z_N\in \R^{dN}\times \R^{dN}$.

In \cite{BCR},
 in the framework of the mean-field theory,
 the rigorous derivation of \eqref{eq:principale} starting from
 \eqref{eq:agenti}
is obtained.
More precisely, the following theorem is proved.

\begin{thm}{\cite[Theorems 3.4, 4.3 and 5.2]{BCR}}
  \label{teo:bcr}

  Let
  $K\colon [0,1] \to \R^+$
  in \eqref{eq:interazione}
  be a positive Lipschitz continuous non-increasing
  function.
  It holds that:
  \begin{enumerate}[i)]
  \item except for a set of initial data 
      $Z_N \in \R^{dN}\times \R^{dN}$ with Lebesgue measure zero,
      there exists a unique global in-time solution
  $$\left(X_N(t,Z_N),V_N(t,Z_N)\right)
  \in C^1(\R^+,\R^{2dN})\times C(\R^+,\R^{2dN})$$
  of \eqref{eq:agenti}.
  Moreover, if $|x_i|\le R_x$ and $|v_i|\le R_v$
  for any $i\in \{1,\dots N\}$,  
  we have that
\begin{equation}
  \label{eq:supporto-part}
  |x_i(t)| \le R_x+tR_v, \ |v_i(t)| \le R_v, \quad i \in \{1,\dots,N\}.
\end{equation}
 \item Let $f_0(x,v)\in L^\infty(\R^d \times \R^d)$ be a
  probability density such that 
  $\mathrm{supp}(f_0)\subset \bar B_{R_x} \times  \bar B_{R_v}$.
  Given $T>0$,
  there exists a unique 
  solution, in the weak sense defined by
  \eqref{pushforward} and \eqref{eq:flusso},
  of the kinetic equation \eqref{eq:principale} with initial datum $f_0$,
 such that
 $f\in C_w\left([0,T];L^{\infty}(\R^d\times \R^d)\right)$,
 \ie $f_t$ is a probability measure
 with bounded density for any
 $t\ge 0$, weakly continuous in $t$.

  Moreover, for any $t\ge 0$, 
    \begin{equation}
    \label{eq:supporto}
    \mathrm{supp}(f_t)\subset \bar B_{R_x+tR_v}\times \bar B_{R_v}
   \ \ \text{ and } \ 
   \| f_t\|_\infty \le \|f_0\|_\infty \e^{d \gamma t},
 \end{equation}
where $\gamma=\int_0^1 K(z) \, \de z$. 

\item If $f_0$ is as in item ii),
    for any $t>0$ 
    the empirical measure
    $\mu_{Z_N(t)}$ associated to the
    system \eqref{eq:agenti} 
    weakly converges to $f_t$ in the limit $N\to +\infty$,
    provided this is true at time zero and
    the initial data $Z_N$ are chosen
    such that the dynamics exists, as in item i).
\end{enumerate}
\end{thm}

In the present work, we
focus on the statistical description of the dynamical system \eqref{eq:agenti}, 
considering, at time $0$, $N$
 independently and identically distributed particles
with law
$F_N(0,Z_N)= f_0^{\otimes N}$.
At time $t$ the particles are distributed
with the law
$F_N(t,Z_N)$,
weak solution of the $N$-body 
Liouville equation
\begin{equation}
\label{Liouville}
\begin{aligned}
&\partial_t F_N(t,Z_N) + \sum_{i=1}^N v_i \cdot \nabla_{x_i} F_N(t,Z_N) \\
+& 
\sum_{i=1}^N \nabla_{v_i} \cdot\Big ( \frac{1}{N} \sum_{j=1}^N 
K\Bigl(M[S\mu_{Z_N}](x_i,|x_i-x_j|)\Bigr) (v_j-v_i) F_N(t,Z_N) \Big)=0,
\end{aligned}
\end{equation}
in the sense that $F_N(t,Z_N)$ is the push-forward of
  $F_N(0,Z_N)$ along the flow $Z_N(t,Z_N)$.
Note that $F_N(t,Z_N)$, for $t > 0$, is symmetric in the exchange of particles.

  In the next section we show that, in the limit
  $N\to +\infty$, for any $s\ge 1$ the $s$-particles marginals of $F_N(t,Z_N)$
  factorize, and the limit
  is described by the solution of the mean-field equation
  $f_t$ with initial datum $f_0$.
  To quantify the convergence, we use the Wasserstein distance $\wass_1$
  defined as follows. 
  Let
  $\mu, \nu \in \mathcal{P}(\R^k)$ be two probability measures
  with finite first moments,
  \begin{equation}
    \label{def:wass}
    \wass_1(\mu, \nu)\coloneqq
    \inf_{\pi \in \mathcal{C}(\mu, \nu)} \int_{\R^k\times \R^k}
    |x-y|
    \de \pi(x,y),
  \end{equation}
  where  $\mathcal{C}(\mu, \nu)$ is the set of all couplings
  of $\mu$ and $\nu$, 
  \ie probability measures on the product space with marginals 
  $\mu$ and $\nu$ in the first and second variable, respectively.

\section{Convergence of the marginals}
\label{sec:3}
In this section we prove the propagation of chaos for solutions of the $N$-body Liouville equation \eqref{Liouville}.
We briefly explain the meaning of this expression and we refer to \cite{CD1, CD2} for a review, and to \cite{NS} for a propagation of chaos result in the case of non-topological Cucker-Smale models.

 We introduce the $s$-particles marginals 
 as follows, with 
an abuse of notation regarding the order of the space and
  velocity variables
  in $F_N(Z_N)$:
\begin{equation*}
F_{N : s} (Z_s) =\int F_N (Z_s, z_{s+1}, \dots, z_N) dz_{s+1}\dots dz_N, \qquad s=1,2, \dots, N,
\end{equation*}
where $z_i = (x_i,v_i)\in \R^{2d}$.
We expect that, if $N$ is large, the details of the individual mutual interactions are negligible,
and the description given by $F_{N :s}(t,Z_s)$ is similar
to the one given by $f_t^{\otimes s}$, where $f_t$ weakly solves 
\eqref{eq:principale}. Note that $f_t^{\otimes s}(Z_s)$ weakly
solves
$$
  \partial_t f_t^{\otimes s}(Z_s) +
  \sum_{i=1}^s v_i \cdot \nabla_{x_i}
  f_t^{\otimes s}(Z_s) +
  \sum_{i=1}^s \nabla_{v_i} \cdot \Big ( W[Sf_t, f_t](x_i,v_i) f_t^{\otimes s}
  (Z_s) \Big )=0.
$$
This last equation describes the law of the system when each particle evolves independently from the others,
with the interaction given by the mean-field
force defined in \eqref{eq:interazione}.
Then propagation of chaos holds if 
$F_{N:s}(t,Z_s)$ converges to $f^{\otimes s}_t(Z_s)$
for any $s\ge 1$.

We will prove the following result.
\begin{thm}[Propagation of chaos for the topological CS model]
  \label{thm1} 
 Assume that the interaction function $K$ is
 as in Theorem \ref{teo:bcr},
  namely a positive Lipschitz continuous non-increasing
  function.
  Let $f \in C_w([0,T]; L^\infty(\R^{2d}))$ be a weak solution
  of the kinetic equation \eqref{eq:principale} with initial datum
  $f_0(x,v)\in L^\infty(\R^d \times \R^d)$ such that $\mathrm{supp}(f_0)
  \subset \bar B_{R_x} \times\bar B_{R_v}$. 
  Consider $F_N(t)$ weak solution of
  \eqref{Liouville} such that $F_N(0)=f_0^{\otimes N}$.

For any integer  $s\in\{1,\dots,N\}$, 
it holds that 
\begin{equation*}
\sup_{t\in [0,T]}
\wass_1 (F_{N:s}(t), f ^{\otimes s}(t)) \le s
 \e^{\lambda(T)\mathcal{K}(1+
   \|f_0\|_{\infty}\e^{d\gamma T})}
 \sqrt{C_{d}(N)},
\end{equation*}
where $\gamma\coloneqq\int_0^1 K(z) \de z$, 
$\lambda(T)$ is a constant depending on $d, R_x, R_v$ and $T$,
and $\mathcal{K} \coloneqq \max (1,\mathrm{Lip}(K),\|K\|_\infty)$, 
where $\mathrm{Lip}(K)$ is the Lipschitz constant of $K$, 
while
\be
\label{fournier}
C_d(N)\coloneqq
\begin{cases}
N^{-1/2} \quad &\text{if} \quad d=1 \\
N^{-1/2}{\log(N)} \quad &\text{if} \quad d=2 \\
N^{-1/d} \quad &\text{if} \quad d>2.
\end{cases}
\ee

\end{thm}

\begin{oss}
  In the proof we will
  need a quantitative version of
  the law of large numbers, \ie
  an estimate of
  $\mathbb{E}[\wass_1(\rho, \mu_{X_N})]$ when $X_N=(x_1, \dots, x_N)$
  are $N$ independently and identically distributed $\R^d$-valued
  random variables with law $\rho \in L^{\infty}(\R^d)$.
  This is a widely studied problem in probability, in connection
  with topics of optimal transport, random matching and combinatorial
  optimization
  (see, for instance, the bibliographical notes to Chapter
  6 of \cite{villani}).
  Here we use the explicit bounds reported by Fournier and Guillin in
  \cite{FG}, which imply that 
  $\mathbb{E}[\wass_1(\rho, \mu_{X_N})] \le c C_d(N)$,
  where $C_d(N)$ is defined in \eqref{fournier} and $c$
  is a constant depending polynomially on the diameter of 
  $\mathrm{supp}(\rho)$ and on $d$. 
\end{oss}

\begin{oss}
  The time dependence of the constant $\lambda(T)$ in Theorem \ref{thm1} comes
  from the growth of the support of the distribution,
  and therefore grows polynomially in time. In particular,
  it is also possible to prove the validity of the theorem for times $T$
  slowly growing with $N$.
\end{oss}

In the proof of Theorem \ref{thm1}
we will employ as a technical tool the
discrepancy distance, defined as follows:
\begin{equation}
  \label{eq:discr}
\disc(\mu, \nu)\coloneqq \sup_{x, r>0} \Big | \int_{{\bar B}_r(x)} \de \mu -
\int_{{\bar B}_r(x)} \de \nu \Big |,
\end{equation}
for $\mu, \nu$ two probability measures on $\R^d$.
We will need the following results concerning it.

\begin{prop}{\cite[Propositions 2.4 and 2.5]{BCR}}

\begin{enumerate}[i)]
  \item Let $\rho $ and $\nu$ be two probability measures on $\R^d$
  with support in a ball $\bar B_R$ and such that $\rho \in L^\infty(\R^d)$.
  Then
  \be
  \label{eq:dw1}
 \disc (\nu, \rho) \le c \sqrt{R^{d-1} \|\rho\|_{\infty}\wass_1(\nu,\rho)},
  \ee 
  where $c$ is a constant that depends on the dimension $d$.

\item Given  $X_N=(x_1,\dots,x_N)$, $Y_N=(y_1,\dots y_N)$,
  with $|x_i-y_i|\le \delta$ for some $\delta $ and any $i=1,\dots N$,
  consider
  the two empirical measures $\mu_{X_N}$ and $\mu_{Y_N}$.
  Then, for any probability measure $\rho \in L^{\infty}(\R^d)$
  supported on a ball $\bar B_R$, 
  \be
  \label{eq:dmunu}
  \disc (\mu_{X_N},\mu_{Y_N}) \le cR^{d-1}\|\rho\|_{\infty}\delta +
  c \disc (\mu_{Y_N},\rho).
  \ee
\end{enumerate}
\end{prop}

  In the following lemma, we summarize the technical details of the proofs
  concerning the Lipschitz estimates of the interaction function $W$ defined in \eqref{eq:interazione}.

  \begin{lem}
    \label{lemma}
   Given $r_x, r_v>0$,
    let $f_1, f_2 \in  \mathcal P(\R^d\times \R^d)$ be
    two probability
    measures supported in $\bar B_{r_x}\times \bar B_{r_v}$
    and $\rho_1, \rho_2 \in \mathcal P(\R^d)$ two probability
    measures supported in $\bar B_{r_x}$ and such that $\|\rho_2\|_{\infty}
    < +\infty$.
    For any $\xi_1,\xi_2 \in \bar B_{r_x}$ and $\eta_1,\eta_2 \in  \bar B_{r_v}$
    it holds:
    \begin{equation*}
      \begin{aligned}
        |W&[\rho_1,f_1](\xi_1,\eta_1) - W[\rho_2,f_2](\xi_2,\eta_2)| \le  
        2r_v \mathrm{Lip}(K)\disc(\rho_1,\rho_2) + \\
        &(c\mathrm{Lip}(K)\|\rho_2\|_{\infty} r_x^{d-1} r_v +
        \|K\|_{\infty})\left( \wass_1(f_1,f_2) + |\xi_1-\xi_2|+|\eta_1-\eta_2|\right),
      \end{aligned}
    \end{equation*}
    where $c$ is a constant that depends only on $d$.
  \end{lem}

\begin{proof}[Proof of Lemma \ref{lemma}]
  By the triangle inequality,
  $ |W[\rho_1,f_1](\xi_1,\eta_1) - W[\rho_2,f_2](\xi_2,\eta_2)| $ can be bounded
  by the sum of
  $$
  \begin{aligned}
    A_1 &\coloneqq |W[\rho_1,f_1](\xi_1,\eta_1) - W[\rho_2,f_1](\xi_1,\eta_1) |,\\
    A_2 &\coloneqq |W[\rho_2,f_1](\xi_1,\eta_1) - W[\rho_2,f_2](\xi_1,\eta_1) |,\\
    A_3 &\coloneqq |W[\rho_2,f_2](\xi_1,\eta_1) - W[\rho_2,f_2](\xi_2,\eta_2) |.\\
  \end{aligned}
$$
Estimate of $A_1$:
by the definition of $M$ in \eqref{eq:Mrho} 
and the definition of
discrepancy distance in \eqref{eq:discr}, we have for any $r\ge 0$
$$|M[\rho_1](\xi_1,r) - M[\rho_2](\xi_1,r)|\le \disc (\rho_1,\rho_2).
$$
By the definition of $W$ in \eqref{eq:interazione}
and the hypothesis on
the supports,
we get that $A_1$ is bounded by
$$
A_1 \le 2\text{Lip}(K)r_v \disc (\rho_1,\rho_2).
$$
Estimates of $A_2$ and $A_3$:
it is easy to prove that,
given $x \in \mathbb{R}^d$ and $r_1, r_2>0$, 
\be
\label{Lipy}
\left|M[\rho_2](x, r_1) - M[\rho_2](x,r_2)\right| \le c\|\rho_2\|_{\infty} \left|r_1^d-r_2^d\right|
\ee
and, given $x_1, x_2 \in \mathbb{R}^d$ and $r>0$,
\be
\label{Lipx}
\left|M[\rho_2](x_1, r)- M[\rho_2](x_2, r)\right| \le c d\|\rho_2\|_{\infty} r^{d-1}|x_1-x_2|.
\ee
By \eqref{Lipy} and \eqref{Lipx},
the Lipschitz constants of the function  
$$K\Bigl(M[\rho_2](x,|x-\xi|)\Bigr)(v-\eta)$$
in the variables $x,\xi\in \bar B_{r_x}$ and $v,\eta \in \bar B_{r_v}$
are 
bounded by $c\mathrm{Lip}(K) \|\rho_2\|_{\infty} r_x^{d-1}r_v$
and $\|K\|_{\infty}$, respectively.
Then, by the definition \eqref{def:wass} of Wasserstein distance,
$$A_2 \le (c\mathrm{Lip}(K)
\|\rho_2\|_{\infty} r_x^{d-1}r_v+\|K\|_{\infty})\wass_1(f_1,f_2),
$$
and
$$A_3 \le
 (c\mathrm{Lip}(K)
 \|\rho_2\|_{\infty} r_x^{d-1}r_v+\|K\|_{\infty})(|\xi_1-\xi_2|+|\eta_1-\eta_2|).
 $$
\end{proof}

\begin{proof}[Proof of Theorem \ref{thm1}]
    For any 
    $\Sigma_N=(y_1,\dots, y_N, w_1, \dots, w_N)\in \R^{dN}\times
    \R^{dN},$
    we consider
    $$\Sigma_N(t,\Sigma_N)=(y_1(t),\dots, y_N(t), w_1(t), \dots, w_N(t))
    \in \R^{dN} \times \R^{dN},$$
    where, for $i\in \{1,\dots, N\}$, $(y_i(t),w_i(t))$ has
     initial datum $(y_i,w_i)$ and
    evolves independently with the 
    mean-field interaction:
    \begin{equation}
      \label{eq:intermediate}
      \left\{
        \begin{aligned}
          &\dot y_i(t)=w_i(t)\\
          &\dot w_i(t)= W[Sf_t, f_t](y_i(t), w_i(t)), \quad i=1, \dots, N.
        \end{aligned}
      \right.
    \end{equation}
    We associate to such $\Sigma_N(t,\Sigma_N)$ the empirical measure 
    $\mu_{\Sigma_N(t)}$ as in \eqref{empirica}.
    
    We define a coupling $\pi^N_t$ between 
    $F_N(t)$ and $f_t^{\otimes N}$ in the following way: at time $t=0$, 
    it is given by
    $$
     \pi^N_0(Z_N, \Sigma_N)
      \coloneqq f_0^{\otimes N}(Z_N) \delta(Z_N - \Sigma_N).
$$
For positive times, $\pi^N_t$ is given by the push-forward of $\pi^N_0$
along the product of the flows given by \eqref{eq:agenti} and \eqref{eq:intermediate}, \ie for any $\varphi \in C_b(\R^{2dN}\times \R^{2dN})$
  $$\int \varphi(Z_N,\Sigma_N)  \de \pi_t^N (Z_N, \Sigma_N) =
  \int \varphi(Z_N(t,Z_N),\Sigma_N(t,Z_N)) \, f_0^{\otimes N}(Z_N) \de Z_N.
  $$
  
Next, given $i \in \{1, \dots, N \}$, we introduce the quantity
\[
 D_N(t)\coloneqq \int \Big(|x_i - y_i| + |v_i - w_i| \Big)\de \pi^N_t(Z_N, \Sigma_N),
\]
which does not depend on $i$, thanks to the symmetry of the law.
We prove
the weak convergence of the $s$-marginals $F_{N:s}$ to $f^{\otimes s}$
by showing that
$D_N(t)\to 0$:
namely, using the symmetry of $\pi_t^N$,
\begin{align*}
  \wass_1(F_{N:s}(t), f_t^{\otimes s})\le
  \sum_{i=1}^s \int \Big(|x_i-y_i| + |v_i - w_i|\Big) \de
  \pi_t^{N}(Z_N, \Sigma_N)\le sD_N(t),
\end{align*}
where we used that
$
| Z_s - \Sigma_s| \le \sum_{i=1}^s (|x_i - y_i|_d + |v_i-w_i|_d)$.

\bigskip 

From the definition of $\pi_t$,
we have, for any $ i \in \{1, \dots, N\}$,
\[
D_N(t)=\int \Big(|x_i(t) - y_i(t)| + |v_i(t) - w_i(t)|\Big)\de f_0^{ \otimes N}(Z_N),
\]
where $(x_i(t), v_i(t))$ for $i=1,\dots N$ solves
  \eqref{eq:agenti} with initial datum $Z_N$,
  and $(y_i(t), w_i(t))$ for $i=1,\dots N$ solves the
  decoupled system \eqref{eq:intermediate} with the same
  initial datum $Z_N$.
It follows that
\begin{equation}
  \label{Ddelta}
D_N(t) \le \int \delta(t,Z_N) \de f_0^{ \otimes N}(Z_N),
\end{equation}
with $\delta(t,Z_N) :=
\text{max}_{i=1, \dots, N} (|x_i(t)-y_i(t)| + |v_i(t) - w_i(t)|)$.

Since $(x_i(t), v_i(t))$ and $(y_i(t), w_i(t))$ have the same initial conditions, it holds that
\begin{align*}
  |x_i(t) - y_i(t)|
  &+ |v_i(t) - w_i(t)| \le \int_0^t |v_i(\tau) - w_i(\tau)| \de \tau
  \\
  &+ \int_0^t \Big|W[S\mu_{Z_N(\tau)},\mu_{Z_N(\tau)}]
    (x_i(\tau),v_i(\tau))  - W[Sf_\tau, f_\tau]
    (y_i(\tau), w_i(\tau)) \Big|\de \tau.
\end{align*}
By Lemma \ref{lemma} with $r_v = R_v$,  $r_x = R_x+\tau R_x$,
we bound the last integrand by
$$\begin{aligned}
    &2R_v \mathrm{Lip}(K) \disc (S\mu_{Z_N(\tau)}, Sf_\tau) +\\
    &(c\mathrm{Lip}(K)\|Sf_\tau\|_{\infty} (R_x+\tau R_v)^{d-1} R_v +
  \|K\|_{\infty})\big( \wass_1(\mu_{Z_N(\tau)},f_\tau)+ \delta(\tau,Z_N)\big).
\end{aligned}
$$
By choosing the coupling $\pi =
\frac 1N \sum_{i=1}^N \delta_{x_i(\tau)}\delta_{v_i(\tau)}\delta_{y_i(\tau)}\delta_{w_i(\tau)}$ in definition \eqref{def:wass}, we can estimate
the Wasserstein distance between the two empirical
measures $\mu_{Z_N(\tau)}$ and $\mu_{\Sigma_N(\tau)}$
by $\delta(\tau,Z_N)$, 
so that 
$$\wass_1(\mu_{Z_N(\tau)},f_\tau) \le \delta(\tau,Z_N) + 
\wass_1(\mu_{\Sigma_N(\tau)},f_\tau) .$$
By the triangle inequality
$$\disc (S\mu_{Z_N(\tau)},Sf_\tau) \le
\disc (S\mu_{Z_N(\tau)},S\mu_{\Sigma_N(\tau)})+
\disc (S\mu_{\Sigma_N(\tau)},Sf_\tau),
$$
and, by \eqref{eq:dmunu} with
$\rho =Sf_\tau$, 
$$\disc (S\mu_{\Sigma_N(\tau)}, S\mu_{Z_N(\tau)})\le 
 c(R_x + \tau R_v)^{d-1} \|Sf_\tau\|_{\infty}\delta(\tau,Z_N)  +
 c \disc (Sf_\tau, S\mu_{\Sigma_N(\tau)}),$$
   where, by \eqref{eq:dw1},
   $$\disc (Sf_\tau,S\mu_{\Sigma_N(\tau)})\le c
   \sqrt{
     (R_x + \tau R_v)^{d-1}\|Sf_\tau\|_{\infty}
     \wass_1(Sf_\tau,S\mu_{\Sigma_N(\tau)})}.$$

   In the sequel we indicate by $\lambda(\tau)$ any positive, increasing
   polynomial function of $\tau$, that
   depends on $d$, $R_v$, $R_v$, \ie on the support
   of $f_\tau$ and $Sf_\tau$, and with $c$ any constant that depends
   at most on the dimension $d$.

Collecting the previous estimates, and using that
$x^{1/2} \le (1+x)/2$ for $x\ge 0$, we get, for a suitable $\lambda(\tau)$,
  \begin{align*}
    &\delta(t,Z_N) 
      \le \\
    &\int_0^t \mathcal {K} \lambda(\tau)
      (1+\|Sf_\tau\|_\infty) \Big( \delta(\tau,Z_N) + 
    \wass_1(\mu_{\Sigma_N(\tau)},f_\tau)
    + \sqrt{\wass_1(S\mu_{\Sigma_N(\tau)},S f_\tau)}\Big) \de \tau,
  \end{align*}
where we used that $\mathcal{K} = \max (1,\mathrm{Lip}(K),\|K\|_\infty)$.
Note that $\|Sf_\tau\|_{\infty}\le cR_v^d \|f_\tau\|_{\infty}$ and,
by Theorem \ref{teo:bcr},
$\|f_\tau\|_\infty \le \e^{d \gamma \tau} \|f_0\|_\infty$.
By the Grönwall's lemma and the fact that
  $x< \e^{x}$, 
  \begin{equation}
    \label{eq:deltafinale}
\delta(t,Z_N) \le
\e^{\mathcal {K}\lambda(T)(1+\|f_0\|_{\infty}\e^{d\gamma T})}
\int_0^t \Big(  
\wass_1(\mu_{\Sigma_N(\tau)},f_\tau) +
\sqrt{\wass_1(S\mu_{\Sigma_N(\tau)},S f_\tau)}
 \Big)
\de \tau.
\end{equation}
  To estimate \eqref{Ddelta},
  we have to evaluate the expected value of the integrand in
    \eqref{eq:deltafinale} w.r.t.
    the probability measure  $\de f_0^{\otimes N}(Z_N)$.
    Since the empirical measure $\mu_{\Sigma_N(\tau)}$
    is supported in 
    $\Sigma_N(t,Z_N)$, then
by \eqref{pushforward},
\begin{align*}
   \int \wass_1(\mu_{\Sigma_N(\tau)},f_\tau)
   \de f_0^{\otimes N}(Z_N)=
   \int \wass_1(\mu_{Z_N},f_\tau)
   \de f_\tau^{\otimes N}(Z_N) \le \lambda(\tau)
   C_{2d}(N),
 \end{align*}
 where the rate $C_d(N)$, as defined in \eqref{fournier},
 is obtained 
 by using the Fournier and Guillin bound in \cite{FG},
 and $\lambda(\tau)$ grows polynomially in $\tau$.
 By concavity we also obtain 
 \begin{align*}
   \int \sqrt{\wass_1(S\mu_{\Sigma_N(\tau)},Sf_\tau)}
   \de f_0^{\otimes N}(Z_N)
&=
   \int \sqrt{\wass_1(S\mu_{Z_N},Sf_\tau)}
   \de f_\tau^{\otimes N}(Z_N)\\
& \le \lambda(\tau)\sqrt{ C_{d}(N)}.
 \end{align*}
 We finally arrive at
 \begin{align*}
   D_N(t)\le
\e^{
\mathcal{K}\lambda(T)
(1+
   \|f_0\|_{\infty}\e^{d\gamma T})
}
\left(\sqrt{C_d(N)} + C_{2d}(N)\right),
 \end{align*}
from which the thesis follows.
\end{proof}

\section{Euler systems associated to monokinetic initial data}
\label{sec:4}

In this section we study the hydrodynamic description of the topological Cucker-Smale model, by considering the following 
pressureless Euler-type system for $(\rho(t,x), u(t,x)) :
[0,T]\times \R^d \to \R \times \R^d$:
\begin{equation}
  \label{eulereq}
  \left\{
  \begin{aligned}
    &\partial_t \rho(t,x)+ \nabla_x \cdot (\rho(t,x) u(t,x))=0 \\
    &\partial_t u(t,x) + (u(t,x) \cdot \nabla_x)u(t,x) =\\
    &\hskip2cm    \int_{\R^d}
     K\Big(M[\rho(t)](x,|x-y|)\Big)(u(t, y)-u(t, x))\rho(t, y) \de y \\
    &(\rho(0,x), u(0,x))=(\rho_0(x), u_0(x)),
  \end{aligned}
  \right.
\end{equation}
where $K(M[\rho])$ is defined as in \eqref{eq:interazione} and \eqref{eq:Mrho},
and
$(\rho_0(x),u_0(x))$ is a regular compactly supported initial datum.

In the following, given $T>0$, we assume that \eqref{eulereq} admits a unique solution
$(\rho(t,x),u(t,x))$ which is regular and compactly supported for $t \in [0, T]$. We leave this fact as a hypothesis, however see \cite{HKK} for a proof in the case of the non-topological Cucker Smale model (see also \cite{KMT, T, LT}).

It is not difficult to show that, if $(\rho(t,x),u(t,x))$ is a regular solution of \eqref{eulereq}, then
\begin{equation}
\label{monokinetic}
f(t,x,v):= \rho(t,x) \delta(v-u(t, x))
\end{equation}
is a weak solution of the topological Vlasov equation in the sense of 
\eqref{pushforward}, \eqref{eq:flusso} with measure-valued initial datum $f(0,x,v)=\rho_0(x) \delta(v-u_0(x))$: these are called monokinetic solutions.
Note that in this case the field $W[Sf_t,f_t](x,v)$ is regular,
so the flow in \eqref{eq:flusso}
is well-defined.

We will show how to obtain
solutions of the Euler system \eqref{eulereq}
starting from solutions of
the topological Liouville equation \eqref{Liouville}.
The main obstacle is that, as stated in Theorem \ref{teo:bcr}, the
  flow for the particle system
  is not defined for every initial datum, so
in general it is not possible to consider a measure-valued solution of
the Liouville equation \eqref{Liouville}.
 To overcome this problem we consider a regularization of the monokinetic initial datum:
\begin{equation}
\label{regmono}
f_0^\eps(x,v)=(\rho_0(x)\delta(v-u_0(x)))*_{x,v}\eta_\eps(x,v),
\end{equation}
where $\eta$ is a $C^{\infty}(\R^{2d})$ compactly supported mollifier and $\eta_\eps(z)=\eta(z/\eps)/{\eps}^{2d}$.

Our next result shows that after relaxing the
regularity of the initial monokinetic measure
 as in \eqref{regmono},
in the limit $N\to +\infty$, uniformly in $\eps$, 
the marginals of solutions of the Liouville equation tend, in
weak sense,
to the tensorial powers of a
monokinetic measure \eqref{monokinetic}, built up out of the solution of
the associated Euler system.

\begin{thm}[From particles to Euler]\label{partoeuler}
  Let us consider a regular solution $(\rho(t,x),u(t,x)),
  \ t\in[0,T]$, of the Euler system \eqref{eulereq} with
  regular compactly supported initial data $(\rho_0(x,v),u_0(x,v))$.
  We suppose moreover that, 
   for any $t\in [0,T]$
   the supports of $\rho(t,x)$ and $u(t,x)$
   are contained in $\bar B_{R_x +t R_v}$
   and that $\|u(t,x)\|_\infty \le R_v$, for some $R_x, R_v>0$.
 Let 
  $F^{\eps}_N(t)$ be a weak
  solution of the Liouville equation \eqref{Liouville}
  with initial datum
  $
  F^\eps_N(0):=(f_0^\eps)^{\otimes N},
  $
with $f_0^\eps$ as in \eqref{regmono}.

Then, for any $t\in[0,T]$ and any $s \in \{1, \dots, N\}$, setting
$V_s=\left (v_1, \dots, v_s\right)$, we have
$$
\wass_1\left(\int_{\R^{ds}}F^\eps_{N:s}(t)  \de V_s,\rho(t)^{\otimes s}\right)
\le s\e^{ \mathcal {K} \lambda(T) \int_0^T (1+\|\rho_\tau\|_\infty) \de \tau}\left (\sqrt{C_{d}(N)}+ \sqrt[4]{\eps}\right),
$$
and, for any Lipschitz function $\Phi:\R^{sd}\times \R^{sd}\to \R$,
setting $X_s=\left (x_1, \dots, x_s\right)$,
\begin{align*}
  &\Bigg | \int \de X_s\de V_s \Phi(X_s,V_s) F^\eps_{N:s}(t,X_s,V_s)
  - \int \de X_s \prod_{i=1}^s \rho(t,x_i)\Phi(X_s,u_1(t,x_1)\dots u_s(t,x_s))
\Biggr| \\
  &\le \mathrm{Lip}(\Phi)
    s\e^{ \mathcal {K} \lambda(T) \int_0^T (1+\|\rho_\tau\|_\infty) \de \tau}\left (\sqrt{C_{d}(N)}+ \sqrt[4]{\eps} \right),
\end{align*}
where $\lambda(T)$ is a constant depending on $d, R_x,R _v, T$ and
$C_d(N)$ is defined in \eqref{fournier}.

\end{thm}

\begin{oss}
Note that the statement doesn't involve the Vlasov equation and that $\eps$ and $N$ are independent,
so the convergence holds also 
for $N\to\infty$ and any sequence
  of limiting points
  as $\eps\to 0$ of the distributions $\{F^\eps_N(t)\}_{N\in \N}$.
  In particular the rate of convergence is $O(\sqrt{C_d(N)})$
  along any sequence with $\eps = \eps_N=O(C_d(N))^2$.
  Note the algebraic behavior in $N$, in contrast with the logarithmic one,
  when the dynamics is mollified as,
e.g., in \cite{PT, BGSR}.
\end{oss}

\begin{oss}
Along the same lines, we can also generalize the content of Theorem \ref{partoeuler}, proving a convergence result for marginals $F^\eps_{N:s}$ of the Liouville equation
with general initial conditions $(g_0^\eps)^{\otimes N}$ such that $\wass_1\left(g_0^\eps,\rho_0(x)\delta(v-u_0(x))\right)\to 0$ as $\eps \to 0$.
\end{oss}

To prove Theorem \ref{partoeuler}, we need the following stability result for solutions of the Vlasov system.

\begin{prop}
\label{propstability}
Given $T>0$, for $i\in\{1,2\}$, let $f^{\text{i}}\in C_w([0,T], \mathcal{P}(\R^d \times \R^d))$ be two weak solutions of the topological Vlasov equation \eqref{eq:principale} with initial data $f^{\text{i}}_0 \in \mathcal{P}(\R^d \times \R^d)$ such that $Sf^{\text{i}}_t$ are well-defined and belong to $L^\infty(\R^d)$. 
Assume moreover that
$
\mathrm{supp} (f^{\text{i}}_t)\subset \bar B_{R_x+ R_v t} \times \bar B_{R_v}
$
for $t>0$ and $R_x, R_v>0$.

Then, for $t \in [0,T]$,
\begin{equation}
\label{stabilitywass}
\wass_1 \left( f^1_t, f^2_t \right) \le
\e^{
  \mathcal {K}  \lambda(T)\min_{i=1,2} \int_0^T(1 +\|Sf^i_\tau\|_\infty)\de \tau }
\max\left \{ \wass_1(f^1_0, f^2_0), \sqrt{\wass_1(f^1_0, f^2_0)} \right \},
\end{equation}
where $\mathcal {K} \coloneqq \max (1, \mathrm{Lip}(K),\|K\|_{\infty})$
and $\lambda(T)$ is a constant depending on $d, R_x, R_v$ and $T$.
\end{prop}

The proof, which requires a nontrivial extension of inequality \eqref{eq:dmunu}, is given in the Appendix.

\begin{proof}[Proof of Theorem \ref{partoeuler}]
  To prove the thesis, it is sufficient to establish the following estimate
  for $t \in [0,T]$:
\begin{align}
  \label{monokinest}
  \wass_1(F^\eps_{N:s}(t), f_t^{\otimes s})\le s\e^{\mathcal {K} \lambda(T)
  \int_0^T(1+\|\rho(\tau)\|_\infty) \de \tau}\left (
  \sqrt{C_d(N)}+ \sqrt[4]{\eps} \right),
\end{align}
where $f$ is the monokinetic solution \eqref{monokinetic}
of the Vlasov equation associated
to the solution $(\rho(t), u(t))$ of the Euler system \eqref{eulereq} and
$\lambda(T)$ is a constant that
depends on $d$, $R_v$, $R_x+T R_v$.

Let $f^\eps(t)$ be the solution of the topological Vlasov equation \eqref{eq:principale} with
initial datum $f_0^\eps$.
By the triangle inequality, we have
$$\wass_1\Big(F^\eps_{N:s}(t), f_t^{\otimes s}\Big) \le 
\wass_1\Big(F^\eps_{N:s}(t), (f_t^\eps)^{\otimes s}\Big) + s
\wass_1\Big(f^\eps_t, f_t\Big),$$
where we are using that $\wass_1((f_t^\eps)^{\otimes s}, f_t^{\otimes s})
\le s \wass_1(f_t^\eps, f_t)$.

The second term is managed by the stability estimate
\eqref{stabilitywass} in Proposition \ref{propstability},
which gives
  \begin{equation}
    \label{oeps}
    \begin{aligned}
      \wass_1
      \Big(
      f^\eps_t, f_t
      \Big) &\le  
      \e^{
       \mathcal {K} \lambda(T)
          \int_0^T(1+\|\rho(\tau)\|_\infty) \de \tau}
      \max\left \{
        \wass_1(f^\eps_0, f_0), \sqrt{\wass_1(f^\eps_0, f_0)} \right \}.\\
\end{aligned}
\end{equation}
By choosing 
$\pi(\de z,\de z') = \eps^{-2d}\eta((z-z')/\eps) f_0(z) \de z \de z'$
in the definition \eqref{def:wass} of the Wasserstein distance, we
can estimate $\wass_1 (f^\eps_0,f_0)$ with $\eps\int |z| \eta(z) \de z$.
Then for $\eps < 1$,
we estimate the maximum in \eqref{oeps}
by $c\sqrt{\eps}$, where $c$ depends only on $\eta$.

The estimate of the first term is similar to the one given in the proof of
Theorem \ref{thm1}. In this case we have
\begin{equation*}
  \wass_1\Big(F^\eps_{N:s}(t), (f_t^\eps)^{\otimes s}\Big)
  \le s\int_{\R^{2dN}} \delta^\eps(t, Z_N)\de (f^\eps_0)^{ \otimes N}(Z_N),
\end{equation*}
where
$\delta^\eps(t,Z_N) \coloneqq
  \text{max}_{i=1, \dots, N} (|x_i(t)-y^\eps_i(t)| +
|v_i(t) - w^\eps_i(t)|)$,
with the independent flow
$\Sigma^\eps(t)\coloneqq (y^\eps_i(t), w^\eps_i(t))=(y^\eps_i(t,Z_N), w^\eps_i(t,Z_N))$
solving
\begin{equation*}
  \left\{
    \begin{aligned}
      &\dot y^\eps_i(t)=w^\eps_i(t)\\
      &\dot w^\eps_i(t)= W[Sf^\eps_t, f^\eps_t](y^\eps_i(t), w^\eps_i(t)), \quad i=1, \dots, N.
    \end{aligned}
  \right.
\end{equation*}
This time we get
\begin{align*}
  |x_i(t) &- y^\eps_i(t)|
  + |v_i(t) - w^\eps_i(t)| \le \int_0^t |v_i(\tau) - w^\eps_i(\tau)| \de \tau
  \\
  &+ \int_0^t \Big|W[S\mu_{Z_N(\tau)},\mu_{Z_N(\tau)}]
    (x_i(\tau),v_i(\tau))  - W[Sf^\eps_\tau, f^\eps_\tau]
    (y_i^\eps(\tau), w_i^\eps(\tau)) \Big|\de \tau.
\end{align*}
In order to avoid terms in $\|Sf_\tau^\eps\|_{\infty}$, which
could diverge for $\eps\to 0$, 
we use carefully the triangle inequality in the last integrand,
which we 
bound by 
$$
\begin{aligned}
  &|W[S\mu_{Z_N(\tau)},\mu_{Z_N(\tau)}](x_i(\tau),v_i(\tau))
  - W[\rho_\tau,\mu_{\Sigma^\eps_N(\tau)}](x_i(\tau),v_i(\tau))|+ \\
  &| W[\rho_\tau,\mu_{\Sigma^\eps_N(\tau)}](x_i(\tau),v_i(\tau))-
  W[Sf^\eps_\tau, f^\eps_\tau]
  (y_i^\eps(\tau), w_i^\eps(\tau))|.
\end{aligned}
$$
By using that 
$\wass_1(\mu_{Z_N(\tau)}, \mu_{\Sigma^\eps_N(\tau)}) \le \delta^\eps(\tau, Z_N)$
and $\mathcal{K}=\max\{1, \text{Lip}(K), \|K\|_\infty\}$,
by Lemma \ref{lemma},
we estimate the first term with
\begin{equation*}
  2R_v\mathcal K \disc(S\mu_{Z_N(\tau)},\rho_\tau) + \mathcal {K}\lambda(\tau) 
\delta^\eps(\tau,Z_N),\end{equation*}
and the second one with
\begin{equation}
  \label{eq:d2}
  2R_v \disc(\rho_\tau,Sf_\tau^\eps) + \mathcal {K}\lambda(\tau) 
  \big(\wass_1(\mu_{\Sigma^\eps_N(\tau)},f^\eps_\tau)+ \delta^\eps(\tau,Z_N)\big).
\end{equation}
We now estimate the discrepancies in the previous two equations.
By the triangle inequality
$$\disc(S\mu_{Z_N(\tau)},\rho_\tau) \le \disc(S\mu_{Z_N(\tau)},S\mu_{\Sigma^\eps_N(\tau)})
+ \disc(S\mu_{\Sigma^\eps_N(\tau)},\rho_\tau),
$$
in which, by inequality \eqref{eq:dmunu}, 
$$\disc(S\mu_{Z_N(\tau)},S\mu_{\Sigma^\eps_N(\tau)})
  \le \lambda(\tau) \|\rho_\tau\|_{\infty}\delta^\eps(\tau,Z_N)+ 
  c\disc(S\mu_{\Sigma^\eps_N(\tau)},\rho_\tau).$$
  We bound the last term
  by using inequality \eqref{eq:dw1}: 
  $$\disc(S\mu_{\Sigma^\eps_N(\tau)},\rho_\tau)\le 
  \lambda(\tau)  (1+\|\rho_\tau\|_{\infty})\big(
  \sqrt{
    \wass_1 (S\mu_{\Sigma^\eps_N(\tau)},S f^\eps_\tau)}
  +
  \sqrt{
    \wass_1 (Sf^\eps_\tau,\rho_\tau)}
  \Big).$$
  Again by inequality \eqref{eq:dw1}, we estimate the
  discrepancy in \eqref{eq:d2} with 
$$\disc( \rho_\tau , Sf_\tau^\eps)\le
\lambda(\tau)(1+ \|\rho_\tau\|_{\infty})
\sqrt{ \wass_1 ( S f_\tau^\eps,\rho_\tau)}.$$
Note that, by \eqref{oeps},
$\wass_1 (S f_\tau^\eps,\rho_\tau)$
is of order $\sqrt{\eps}$.

Collecting these estimates together, and using Grönwall's lemma,
we arrive at
$$
\begin{aligned}
\delta^\eps(t,Z_N) &\le \e^{ \mathcal {K} \lambda(T)\int_0^T (1+ \|\rho_\tau\|_\infty) \de \tau} \\
&\times \Big( \sqrt[4]{\eps} + \int_0^T\de \tau \int \Big(\sqrt{\wass_1 (S f^\eps_\tau,S\mu_{\Sigma^\eps_N(\tau)})}+{\wass_1 (\mu_{\Sigma^\eps_N(\tau)},f^\eps_\tau)}\Big) \de f^{\otimes N}_0(Z_N)  \Big).
\end{aligned}
$$
Using again the Fournier and Guillin bound in
\eqref{fournier} we get the estimate \eqref{monokinest},
from which the thesis follows.

\end{proof}

 \section*{Appendix: proof of Proposition \ref{propstability}}

   We consider, for $t\in [0,T]$,
   two weak solutions $f^1_t$ and $f^2_t$
   of \eqref{eq:principale}
   in the weak sense specified by
   \eqref{pushforward} and \eqref{eq:flusso},
   with support on $\bar B_{R_x+tR_v}\times \bar B_{R_v}$. We indicate by 
   $Z^i(t,z)$, $i=1,2$, the corresponding flows and
   we assume that $\min_{i=1,2} \int_0^T \|Sf_\tau^i\|_{\infty}\de \tau$
   is attained for $i=1$.
   We define the intermediate dynamics
   $(\tilde f(t,z),\tilde Z(t,z))$,
   where $\tilde f(t)$ is the push-forward of $f^2_0$
   by the flow $\tilde Z(t,z)=(\tilde x(t,x,v),\tilde y(t,x,v))$ which solves
   \begin{equation*}
     \left\{
       \begin{aligned}
         &\dot {\tilde x}(t,x,v)=\tilde v(t,x,v)\\ 
         &\dot {\tilde v}(t,x,v)=W[Sf^1_t,\tilde f_t]
         (\tilde x(t,x,v),\tilde v(t,x,v))  \\
         & \tilde x(0,x,v) = x,\ \ \tilde v(0,x,v) = v.\\
       \end{aligned}
     \right.
   \end{equation*}
   From the hypothesis,
   for $i=1,2$,
   the spatial density $Sf^i_t$ is bounded,
   and then,
  by \eqref{Lipy} and \eqref{Lipx}, the field
  $K\left(M[Sf^i(t)](x,|x-y|)\right)$ is locally Lipschitz in $x$ and $y$.
  Hence $(Z^i(t, z))$, $i=1,2$, exists for any
  $t\in [0,T]$, and also the couple
  $(\tilde f(t, z), \tilde Z(t,z))$ is well-defined.

As in previous sections, we indicate by $\lambda(\tau)$
any constant that
depends on $d$, $R_v$, $R_x+\tau R_v$. We have
\begin{equation}
\label{ineqApp}
\wass_1\left(f^1_t, f^2_t \right) \le  \wass_1\left(f^1_t, \tilde f_t\right) +
\wass_1\left(\tilde f_t, f^2_t \right).
\end{equation}
The first term is under control since the continuity estimate
  à la Dobrushin holds. Namely, from Lemma \ref{lemma},
  \begin{equation*}
    | W[Sf^1_\tau,f^1_\tau](x,v) -W[Sf^1_\tau,\tilde f_\tau](x,v)| \le
    \mathcal {K}
    \lambda(\tau)
    (1+ \|Sf^1_\tau \|_\infty)
  \wass_1(f^1_\tau,\tilde f_\tau),
  \end{equation*}
  from which 
  $$
  \begin{aligned}
    |Z^1(t,z)-\tilde Z(t,\tilde z)| &\le |z-\tilde z|
    + \\
    &\int_0^t \mathcal {K}  \lambda (\tau)  (1+\|Sf_\tau^1\|_{\infty})
    \big(
    \wass_1(f^1_\tau,\tilde f_\tau)+|Z^1(\tau,z) - \tilde Z(\tau,\tilde z)|
    \big) \de \tau,
    \end{aligned}$$
  which allows us to obtain
  \begin{equation}
    \label{stimawass1}
    \wass_1\left(f^1_t, \tilde f_t\right) \le \e^{\lambda(T)
     \mathcal {K} \int_0^t(1+\|Sf^1_\tau\|_\infty) \de \tau}
    \wass_1 \left( f^1_0, f^2_0 \right), \quad
    t\in [0,T].
\end{equation}

Concerning the second term in the inequality \eqref{ineqApp}, we have
$$  \wass_1(\tilde f_t, f^2_t)
\le \int_{\R^{2d}} | Z^2(t,z) - \tilde Z(t,z)|
\de f_0^2(z)
 \le \delta(t)
\coloneq \sup_{z \in \mathrm{supp}(f_0^2)} |Z^2(t,z) - \tilde Z(t,z)|.
$$
To estimate  $\delta(t)$, we note that
\[
  |Z^2(t,z) - \tilde Z(t,z)| \le \int_0^t \Big{(}\delta(\tau) +\left
    |W[Sf^2_\tau, f^2_\tau](Z^2(\tau,z))-
    W[ Sf^1_\tau, \tilde f_\tau](\tilde Z(\tau,z))\right |\Big{)} \de \tau.
\]
Using Lemma \ref{lemma},
we estimate the last term in the integrand by
\begin{equation}
  \label{stimaapp2}
  2R_v\mathcal K \disc(Sf^2_\tau, Sf^1_\tau)+\mathcal {K}\lambda(\tau)
  (1+\|Sf^1_\tau\|_\infty)
  (\wass_1(f^2_\tau,\tilde f_\tau) + \delta(\tau)).
\end{equation}

Hence we need only to estimate $\disc( Sf^2_\tau, Sf^1_\tau)$. By the triangle inequality
$$\disc(Sf^2_\tau, Sf^1_\tau) \le \disc(Sf^2_\tau, S\tilde f_\tau) +
\disc(S\tilde f_\tau, Sf^1_\tau).
$$
The second term can be bounded by using inequality \eqref{eq:dw1},
obtaining that
\begin{equation}
\label{stimaapp3}
\disc(S\tilde f_\tau, Sf^1_\tau)\le c \sqrt{\|Sf^1_\tau\|_{\infty}(R_x + \tau R_v)^{d-1}\wass_1(S\tilde f_\tau, S f^1_\tau)}.
\end{equation}

We now establish an estimate for $\disc(Sf^2_\tau, S\tilde f_\tau)$
of the type of inequality  \eqref{eq:dmunu}.
For any $z_0,z\in \R^{2d}$, by definition of $\delta(t)$, 
$$|z_0- Z^2(t,z)|-\delta(t) \le |z_0-\tilde Z(t,z)|
\le |z_0-Z^2(t,z)|+\delta(t). $$
Let $X$ be the set of functions
$\phi\in C([0,+\infty),\R)$, with first derivative
continuous up to a finite number of jumps. 
Given $\phi\in X$ and $\delta >0$, let 
$
\phi_\delta(r)  \coloneq \phi^+ (r+\delta) - \phi^-(r-\delta),
$
where, denoting by $\tilde \phi$ the function
$
\tilde \phi (r) \coloneq \int_0^r |\phi'(s)|\de s,$
we have defined 
$$\phi^\pm (r)\coloneq \left\{
  \begin{aligned}
    &\frac 12 (\tilde \phi(r) \pm \phi(r)),\  &\text{ if }r\ge 0,\\
    &\pm \frac 12 \phi(0), &\text{ if }r<0.
  \end{aligned}
\right.
$$
It can be proved (see  \cite[Lemma 2.2]{BCR}) that 
$$\disc (\rho_1,\rho_2) = \sup_{\phi \in X: \, \|\phi\|_X \le 1}
\sup_x
\int \phi\bigl(|x -y|\bigr)\bigl(\de\rho_1(y)-\de\rho_2(y) \bigr),$$
where $\|\phi\|_X \coloneq \int_0^{+\infty} |\phi'(r)|\de r.$

Fixed $x_0 \in \R^d$, we define $\Phi(x)=\phi(|x-x_0|)$ and $\Phi_\delta$ in the same
way.
Then, it is not difficult to see that (see 
\cite[Lemma 2.3]{BCR})
$$\Phi \left(X^2(t,z)\right)\le \Phi_\delta \left(\tilde X(t,z) \right).$$
Hence 
$$
\begin{aligned}
  &\int \Phi \left(\de Sf^2_\tau - \de S\tilde f_\tau \right) =
  \int \Phi (x) \left( \de f^2(z) - \de \tilde f(z)\right) \\
  &=
  \int \left(\Phi (X^2(\tau,z)) -\Phi (\tilde X(\tau,z)) \right)
  \de f^2_0(z) \\
  &\le
  \int \left(\Phi_\delta (\tilde X(\tau,z)) -\Phi (\tilde X(\tau,z))\right)
  \de f^2_0(z)  =
  \int \left (\Phi_\delta(x) -\Phi(x) \right) \de S\tilde f_\tau(x)
  \\
  & =   \int \left(\Phi_\delta -\Phi \right)
  \left(\de S \tilde f_\tau - \de Sf ^1_\tau\right) +
  \int \left(\Phi_\delta -\Phi \right) \de Sf^1_\tau.
\end{aligned}
$$
Since $\Phi_\delta - \Phi \in X$, the first term is bounded
by $c\disc ( S\tilde f_\tau,Sf^1_\tau)$, while the second can be easily bounded by
$c(R_x+\tau R_v)^{d-1}\delta(\tau) \|Sf^1_\tau\|_{\infty}.$
We conclude that 
\begin{equation}
\label{stimaapp4}
\disc(Sf^2_\tau,S\tilde f_\tau) \le \lambda(\tau)\delta(\tau)\|Sf^1_\tau\|_{\infty}
+ c \disc ( S\tilde f_\tau,Sf^1_\tau).
\end{equation}
Collecting estimates \eqref{stimaapp2}, \eqref{stimaapp3} and \eqref{stimaapp4}, we arrive at
$$
  \begin{aligned}
    &\delta(t) \le \int_0^t \lambda(\tau) \mathcal {K}
    (1+\|Sf^1_\tau\|_\infty) \Bigl(
    \delta(\tau) + 
    \sqrt{\wass_1(Sf^1_\tau, S\tilde f_\tau)}\Bigr)\de \tau.
    \end{aligned}
$$
Using the Grönwall's lemma and that $x \le \e^{x}$, 
for $t \in [0,T]$,
$$
\begin{aligned}
  \wass_1 \left(
  \tilde f_t, f^2_t\right)& 
  \le
  \delta(t)
   \le
  \e^{\lambda(T)\mathcal {K}\int_0^T(1+\|Sf^1_\tau\|_\infty)\de\tau}
  \\&\times \int_0^T \sqrt{ \wass_1(Sf^1_\tau, S\tilde f_\tau)} \de \tau.
\end{aligned}
$$
The thesis follows from the last inequality, together with \eqref{ineqApp} and \eqref{stimawass1}.

\end{document}